\documentclass[a4paper, 12pt]{amsart}
\usepackage{cite}
\usepackage{amsmath}
\usepackage{amsthm}
\usepackage[all]{xy}
\usepackage{amssymb}
\usepackage{mathrsfs}
\usepackage{etoolbox}
\usepackage{fullpage}

\newtheorem{thm}{Theorem}
[section]
\newtheorem{cor}[thm]{Corollary}

\newtheorem{prop}[thm]{Proposition}

\theoremstyle{definition}
\newtheorem{rem}[thm]{Remark}
\newtheorem{fct}[thm]{Fact}
\newtheorem{defn}[thm]{Definition}

\newtheorem{rmk}[thm]{Remark}
\numberwithin{equation}{subsection}

\usepackage{color}
\begin{document}

\title{ON THE ALGEBRAIC INDEPENDENCE OF  GENERIC PAINLEV\'E TRANSCENDENTS.}
\author{Joel Nagloo$^1$, Anand Pillay$^2$}
\date{September 5, 2012}
\maketitle 
\pagestyle{plain}
\begin{abstract} We prove that if $y''=f(y,y',t,\alpha,\beta,\ldots)$ is a generic Painlev\'e equation from among the classes $II$ to $V$, and if $y_1,\ldots,y_n$ are distinct solutions, then 
\newline
$tr.deg\left(\mathbb{C}(t)(y_1,y'_1,\ldots,y_n,y'_n)/\mathbb{C}(t)\right)=2n$. (This was proved by Nishioka for the single equation $P_{I}$.) 
For generic Painlev\'e VI, we have a slightly weaker result: $\omega$-categoricity (in the sense of model theory) of the solution space, as described below.
\end{abstract}
\footnotetext[1]{Supported by an EPSRC Project Studentship and a University of Leeds - School of Mathematics partial scholarship}
\footnotetext[2]{Supported by an EPSRC grant $EP/I002294/1$}
\section{Introduction}

In this paper we are concerned with algebraic relations over $\mathbb{C}(t)$ between solutions of a generic Painlev\'e equation. We direct the reader to our earlier paper \cite{NagPil} for a very detailed introduction but also for a summary of the model theoretic techniques. We conjectured there that for the generic Painlev\'e equations from each of the families $P_{I}-P_{VI}$ (see the list below), if $y_1,\ldots,y_n$ are solutions viewed as meromorphic functions on some disc $D\subset\mathbb{C}$ and if we work in the differential field $F$ of meromorphic functions on $D$ (which contains the differential subfield $\mathbb{C}(t)$ of rational functions), then $tr.deg\left(\mathbb{C}(t)(y_1,y'_1,\ldots,y_n,y'_n)/\mathbb{C}(t)\right)=2n$, that is $y_1,y'_1,\ldots,y_n,y'_n$ are algebraically independent over $\mathbb{C}(t)$. So there are ``no algebraic relations between distinct solutions (and their derivatives)". $P_{I}$ is a single equation, and the result was proved in this case by Nishioka\cite{Nishioka2}. In this paper we prove the conjecture for the families $P_{II} - P_{V}$. The situation for $P_{VI}$ is more delicate. It may  very well be the case that the conjecture is true there, but all we can prove is the following: given solutions $y_{1},..,y_{k}$ of generic $P_{VI}$ such that $tr.deg(\left(\mathbb{C}(t)(y_{1},y_{1}',\ldots, y_k, y'_k)/\mathbb{C}(t)\right) = 2k$, then for all other solutions $y$, except for at most $11k$, $tr.deg(\left(\mathbb{C}(t)(y_{1},y_{1}',\ldots, y_k, y'_k, y, y')/\mathbb{C}(t)\right) = 2(k+1)$.
(And it is well-known that  for any single solution $y$, $tr.deg(\left(\mathbb{C}(t)(y,y')/\mathbb{C}(t)\right) = 2$.)\\

There are other natural and related questions concerning algebraic relations between solutions of generic Painlev\'e equations from different families, some of which can be treated using methods of this paper, and this is being pursued by the first author.\\

\begin{equation*}
\begin{split}
P_{I}:\;\;\;\;\;  & \frac{d^2y}{dt^2}=6y^2+t \\
P_{II}(\alpha):\;\;\;\;\; & \frac{d^2y}{dt^2}=2y^3+ty+\alpha \\
P_{III}(\alpha,\beta,\gamma,\delta):\;\;\;\;\; & \frac{d^2y}{dt^2}=\frac{1}{y}\left(\frac{dy}{dt}\right)^2-\frac{1}{t}\frac{dy}{dt}+\frac{1}{t}(\alpha y^2+\beta)+\gamma y^3+\frac{\delta}{y} \\
P_{IV}(\alpha,\beta):\;\;\;\;\; & \frac{d^2y}{dt^2}=\frac{1}{2y}\left(\frac{dy}{dt}\right)^2+\frac{3}{2}y^3+4ty^2+2(t^2-\alpha)y+\frac{\beta}{y} \\
P_{V}(\alpha,\beta,\gamma,\delta):\;\;\;\;\; & \frac{d^2y}{dt^2}=\left(\frac{1}{2y}+\frac{1}{y-1}\right)\left(\frac{dy}{dt}\right)^2-\frac{1}{t}\frac{dy}{dt}+\frac{(y-1)^2}{t^2}\left(\alpha y+\frac{\beta}{y}\right)+\gamma\frac{y}{t}\\
& +\delta\frac{y(y+1)}{y-1}\\ 
P_{VI}(\alpha,\beta,\gamma,\delta):\;\;\;\;\; & \frac{d^2y}{dt^2}=\frac{1}{2}\left(\frac{1}{y}+\frac{1}{y+1}+\frac{1}{y-t}\right)\left(\frac{dy}{dt}\right)^2-\left(\frac{1}{t}+\frac{1}{t-1}+\frac{1}{y-t}\right)\frac{dy}{dt}\\
 & +\frac{y(y-1)(y-t)}{t^2(t-1)^2}\left(\alpha+\beta\frac{t}{y^2}+\gamma\frac{t-1}{(y-1)^2}+\delta\frac{t(t-1)}{(y-t)^2}\right)
\end{split}
\end{equation*}
where $\alpha,\beta,\gamma,\delta\in\mathbb{C}$.\\

In \cite{NagPil}, we proved a weak version of this algebraic independence conjecture which is valid for all generic Painlev\'e
equations. Namely we showed that if $y_1,\ldots,y_n$ are distinct solutions and if $y_1,y'_1,\ldots,y_n,y'_n$ are algebraically dependent over $\mathbb{C}(t)$, then already for some $1\leq i<j\leq n$, $y_i,y'_i,y_j,y'_j$ are algebraically dependent over $\mathbb{C}(t)$. This corresponds to the model theoretic notion ``geometric triviality" and the proof consisted of combining results by the Japanese school on ``irreducibility" of the Painlev\'e equations, with nontrivial results in the model theory of differentially closed fields (such as the trichotomy
theorem for strongly minimal sets).\\

The proofs of the main results in the current paper have three ingredients; (i) the ``geometric triviality" results from \cite{NagPil}, (ii) the description/classification  (in the literature) of algebraic solutions of the Painlev\'e equations in the various families, as the parameters vary, depending also on the understanding of the relevant Backlund/Okamoto transformations, and (iii) elementary model theoretic considerations, specifically quantifier elimination for $DCF_{0}$.
So overall  by combining the existing global structural analysis of the Painlev\'e families  (i.e. irreducibility, and existence of algebraic solutions, as parameters vary) with both nontrivial and elementary model theory of differential fields, we obtain definitive information about ``generic" Painlev\'e equations.\\

 
In section 2 we give  precise definitions of the various notions we will be using, concentrating on the case at hand. In the third section we prove the main conjecture for the Painlev\'e equations $P_{II}-P_{V}$ (Propositions \ref{P1}, \ref{P2}, \ref{P3}, \ref{P4}), by first in each case describing the classification of algebraic solutions. In the final section, we deal with $P_{VI}$ which is more delicate; we obtain the weaker statement mentioned above (Proposition \ref{P5}).\\

Many thanks are due to Philip Boalch and Marta Mazzocco for pointing out, and explaining the significance of,  Boalch's ``generic icosahedral solution" to Painlev\'e VI, which is precisely what is needed to prove our results for generic Painlev\'e VI. 
\section{Preliminaries}

Note that apart from $P_{I}$ which is a single equation, each of the other families is parametrized by a tuple $(\alpha, \beta,..)$ of complex numbers. We will say that an equation in one of these families is ``generic" if the corresponding tuple of complex numbers is an algebraically independent tuple of transcendental complex numbers. In fact $P_{III}$ is essentially a $2$-parameter family, and $P_{V}$ a $3$-parameter family (see \cite{NagPil}). \\

The reader is free to view a ``solution" to any of the equations above, as a meromorphic function on some open connected set $D\subseteq \mathbb{C}$. 
The collection of all meromorphic functions on $D$ (equipped with $d/dt$) is a differential field containing $\mathbb{C}(t)$ and as such one can discuss transcendence questions. One the other hand it is natural to think, more generally,  of a solution as  an element $y$ of an arbitrary differential field
$(F, \partial)$ extending $(\mathbb{C}(t),d/dt)$  (which solves the equation in the obvious sense).  This is the point of view of the Japanese school in their study of irreducibility for example, and will also be the point of view of the current paper.\\

In fact it will be important for us to work in the framework of an ambient differentially closed field $\mathcal{U}$, one of the reasons being that the first order theory $DCF_{0}$ of differentially closed fields of characteristic $0$ has quantifier elimination: any first order formula $\phi(x_{1},..,x_{n})$ is equivalent to a quantifier-free formula. The underlying ``language" here is that of differential rings $\{+,-,\cdot,0,1,\partial\}$. We will call this language $L$ when there is no chance for ambiguity. See \cite{Marker} for more details on the model theory of differential fields, in particular the $\omega$-stability of 
$DCF_{0}$. Section 2 of \cite{NagPil} discusses in detail model-theoretic notions and results relevant to the work in this paper. But in fact the current paper will not require so much model theory, mainly just the results (rather than background) from \cite{NagPil}. For a field $L$, $L^{alg}$ denotes its algebraic closure (in the field-theoretic sense). We will discuss later the model-theoretic notion of algebraic closure and its meaning in 
differentially closed fields. But we try to keep this paper relatively self-contained. \\

We will now give the relevant notions, in the special case of equations over $\mathbb{C}(t)$ of the form $y'' = f(y,y')$ with $f$ a rational function (where eventually this will be one of the Painlev\'e equations). We take $(\mathcal{U},\partial)$ to be a ``saturated" (or universal, in the sense of Kolchin) differentially closed field  of cardinality the continuum, and all differential fields we consider will be sub differential fields of $\mathcal{U}$. Without loss of generality the field of constants of $\mathcal{U}$ is $\mathbb{C}$. We take $t\in\mathcal{U}$ with $\partial(t) =1$. Let $F_{0}$ be a finitely generated subfield of $\mathbb{C}$ We fix an $ODE$, $y'' = f(y,y',t)$ where $f$ is a rational function (in $y$, $y'$, $t$) over $F_{0}$. Let $K_{0} = F_{0}(t)$ (field of rational functions over $F_{0}$), and we also let $K$ denote $\mathbb{C}(t)$ (field of rational functions over $\mathbb{C}$). 
Let $X$ be the set of solutions of the equation $y'' = f(y,y',t)$ in $\mathcal{U}$.

\begin{defn} Let $L$ be a differential field containing $K_{0}$ (e.g. $L = K$). We call a solution $y\in X$ {\em generic} over $L$ if $tr.deg\left(L(y,y')/L\right)=2$, and call solutions $y_{1},..,y_{n}\in X$ {\em mutually generic} if $tr.deg\left(L(y_1,y'_1,\ldots,y_n,y'_n)/L\right)=2n$
\end{defn}

\begin{rem} If $L$ is countable then by saturation of $\mathcal{U}$ there will in fact exist a continuum of mutually generic solutions of $X$ over $L$
\end{rem}

\begin{defn}  $X$  (or the equation) is said to be {\em strongly minimal} if for any differential field $L$ containing $K_{0}$, and $y\in X$, either $y\in L^{alg}$, or $y$ is generic over $L$.
\end{defn}
 \noindent
{\em Explanation and commentary.}  This is equivalent to any definable subset of $X$ being finite or cofinite, and the latter is the {\em definition} of a strongly minimal definable set $X$ in an arbitrary structure. For example the set of constants is strongly minimal in $\mathcal{U}$. 
But under the current assumptions on $X$  strong minimality corresponds to the equation $y'' = f(y,y',t)$ satisfying Umemura's $J$-condition. See \cite{NagPil} for more details. Note that if $X$ is strongly minimal, $L > K_{0}$, $y_{1},..,y_{n}\in X$ and $y_{i+1}\notin L(y_{1},y_{1}',..,y_{i},y_{i}')^{alg}$ for $i=1,..,n-1$ then $y_{1},..,y_{n}$ are mutually generic over $L$. \\

\begin{rem} Suppose that $X$ is strongly minimal and $y_{1},..,y_{n}\in X$ are mutually generic over $K_{0}$. Then they are also mutually generic over 
$K = \mathbb{C}(t)$. In particular if $y\in X$ is in $K^{alg}$ then it is already in $K_{0}^{alg}$. 
\end{rem}
\noindent
{\em Explanation.} This is because $X$, being strongly minimal and of order $2$ is ``orthogonal" to the constants $\mathbb{C}$, the latter being a strongly minimal set of order $1$. See \cite{NagPil}. \\

\begin{defn} Suppose $X$ is strongly minimal. We say that $X$ is {\em geometrically trivial} if whenever $y_{1},..,y_{n} \in X$ are pairwise mutually generic over $K_{0}$ then they are mutually generic over $K_{0}$.
\end{defn}
\noindent
{\em Commentary.}  This is equivalent to saying: for any $L> K_{0}$ whenever $y_{1},..,y_{n}\in X$ are pairwise mutually generic over $L$, then they are mutually generic over $L$.\\

\begin{defn} Let $L> K_{0}$. We say that $X$ is {\em strictly disintegrated} over $L$, if whenever $y_{1},..,y_{n}\in X$ are distinct, then they are mutually generic over $L$. 
\end{defn}
\noindent
{\em Commentary.}  Note that strict disintegratedness over $L$ of $X$ implies strong minimality of $X$. It also implies that no solution is in $L^{alg}$.
As in Remark 2.4 we have that $X$ is strictly distintegrated over $K_{0} = F_{0}(t)$ iff it is strictly disintegrated over $K = \mathbb{C}(t)$. 
Finally note that strict disintegratedness of $X$ over $L$ implies that any permutation of $X$ extends to an automorphism of the differential field $\mathcal{U}$ which fixes $L$ pointwise. \\

Now we mention $\omega$-categoricity (which is  what we will prove for generic $P_{VI}$). The notion is treated in some detail in \cite{NagPil}, but here we only consider it in the strongly minimal context. 
\begin{defn} Suppose $X$ is strongly minimal. We say that $X$ is {\em $\omega$-categorical}, if whenever $y_{1},..,y_{k}\in X$ and 
$L$ is the differential field generated by $K_{0}$ and $y_{1},..,y_{k}$, then only finitely many $y\in X$ are in $L^{alg}$. 
\end{defn}
\noindent
{\em Commentary.} As remarked in \cite{NagPil} this definition does not depend on the choice of the finitely generated subfield $F_{0}$ of $\mathbb{C}$ over which $f(y,y',t)$ is defined. Moreover, as in Remark 2.4, we can replace $K_{0}$ by $K = \mathbb{C}(t)$. Clearly if (strongly minimal) $X$ is strictly disintegrated over $K_{0}$, then $X$ is $\omega$-categorical, and in turn $\omega$-categoricity  implies geometric triviality. (See \cite{NagPil} for discussion of the last implication which is rather specific to $DCF_{0}$.)\\ 

Now let us mention the main results from \cite{NagPil} in the light of the above definitions.  For each of the Painlev\'e equations $y'' = f(y,y',t,\alpha, \beta, ..)$ we will take $F_{0}$ to be the subfield of $\mathbb{C}$ generated by the parameters $\alpha, \beta,...$, so 
$K_{0} = \mathbb{Q}(\alpha, \beta,..)(t)$.

\begin{fct} Suppose $y'' = f(y,y',t,..)$ is a generic Painlev\'e equation in any of the classes $I - VI$. Let $X$ be the solution set. Then $X$ is strongly minimal, there are no algebraic (over $K_{0}$, so also over $K$) solutions, and moreover $X$ is geometrically trivial.
\end{fct}
\noindent
{\em Commentary.}  Strong minimality was noted in \cite{NagPil}, as a consequence of the work of the Japanese school on ``irreducibility"  (they classified the parameters for which the corresponding equation has Umemura's J-property, and the latter is equivalent to strong minimality). Likewise the parameters for which there exists algebraic (over $\mathbb{C}(t)$) solutions of the corresponding equations have been classified. References will appear throughout this paper. Geometric triviality was the main result of \cite{NagPil}: 
see Propositions 3.1, 3.6, 3.9, 3.12, 3.15, and 3.18 in that paper.\\

\vspace{5mm}
\noindent
We will prove that generic Painlev\'e equations in class $II - V$ are strictly disintegrated over $K_{0}$, and thus over $K = \mathbb{C}(t)$. By virtue of Remark 2.4, it will suffice to prove that any two solutions $y_{1}, y_{2}\in X$ are mutually generic over $K_{0}$.
Likewise to prove that generic $P_{VI}$ is $\omega$-categorical it will suffice to prove that for a solution $y$ of $X$, there are only finitely many other solutions in $K_{0}(y,y')^{alg}$. \\

\vspace{5mm}
\noindent
Let us finish this section with a few comments on the very basic model-theoretic context, in particular ``algebraic closure" in $\mathcal{U}$, as our proofs will be from this point of view.  Given a structure $M$ for a countable language $L$ (where we identify $M$ notationally with its underlying set or universe), a subset $A$ of (the underlying set of) $M$ and a finite tuple $a$ from $M$, we say that ``$a$ is algebraic over $A$ in $M$, in the sense of model theory", if there is a formula $\phi(x)$ with parameters from $A$ such that $M\models \phi(a)$, and there are only finitely many other tuples $b$ such that $M\models \phi(b)$.  Given no ambiguity about the ambient structure $M$, by $acl(A)$ we mean the set of all tuples which are algebraic over $A$ in $M$. If $X$ is a definable subset of $M^{n}$, defined over $A$, then by $acl_{X}(A)$ we mean $acl(A)\cap X$, the set of elements of $X$ which are algebraic over $A$. It should be mentioned that sometimes $acl(A)$ is used to denote the set of {\em elements} of the structure $M$ which are algebraic over $A$, but there is no real ambiguity as a finite tuple is algebraic over $A$ if and only if each of its coordinates is algebraic over $A$.  We also have the notion $dcl(A)$, definable closure of $A$ in $M$, which is above except $a$ should be the unique realization of $\phi$ in $M$.\\

As in the commentary to Definition 2.3,  a definable set $X$ (in $M$) is strongly minimal if it is infinite and any definable subset is finite or cofinite. Suppose $X$ is strongly minimal and definable over a finite set $A$. Assume some degree of ``saturation" of $M$, as well as stability of $Th(M)$. Then $acl(-)$ has the following exchange property: if $b,c\in X$, $b\notin acl(A)$, and $c\in acl(A,b)$ then $b\in acl(A,c)$. Geometric triviality of $X$ amounts to: whenever
$b_{1},..,b_{n} \in X\setminus acl(A)$, and $b_{i}\notin acl(A,b_{j})$ whenever $i\neq j$, then for each $i=1,..,n$, $b_{i}\notin acl(A,b_{1},..,b_{i-1}, b_{i+1},..,b_{n})$.  And $\omega$-categoricity of $X$ amounts to: for any finite subset $B$ of $X$, $acl_{X}(A\cup B)$ is finite.
\\

Now let us consider our differential closed field $\mathcal{U}$. We have:
\begin{fct} Let $A$ be a subset of $\mathcal{U}$. Then
\newline
(i) $dcl(A)$ is the differential subfield of $\mathcal{U}$ generated by $A$,
\newline
(ii) $acl(A)$ is $dcl(A)^{alg}$ the field-theoretic algebraic closure of $dcl(A)$ (which will also be a differential subfield of $\mathcal{U}$). 
\end{fct}

Reverting to the case where $X$ is the solution set of $y' = f(y,y',t)$ with $f$ rational over $F_{0}$ and $K_{0} = F_{0}(t)$:
\begin{fct} Let $y_1,..,y_n, y \in X$. Then $y$ is generic in $X$ over the (differential) field $K_{0}(y_{1}, y_{1}',..,y_{n}, y_{n}')$ if and only if $y\notin acl(K_{0},y_{1},..,y_{n})$ in the model-theoretic sense.
\end{fct}

\vspace{2mm}
\noindent
A trivial model-theoretic fact we use is expressibility of $acl$:  namely in an arbitrary structure $M$ for language $L$, if $c\in acl(A)$ then there is an $L$-formula $\phi(x,y)$, and tuple $a$ from $A$, such that $M\models \phi(c,a)$, and whenever  $c_{1},a_{1}$ are from $M$ such that $M\models\phi(c_{1},a_{1})$ then $c_{1} \in acl(a_{1})$. Our main argument will combine this with the following equally basic fact about $DCF_{0}$.
\begin{fct}  Let $\phi(x_{1},x_{2},...,x_{n},y)$ be a formula in the language of differential fields. Suppose $\alpha_{1}, \alpha_{2}, ..,\alpha_{n}$ are algebraically independent complex numbers such that $\mathcal{U} \models \phi(\alpha_{1}, \alpha_{2},.., \alpha_{n},t)$. Then for all but finitely many $c\in \mathbb{C}$, we have
$\mathcal{U} \models \phi(c,\alpha_{2},..,\alpha_{n},t)$. 
\end{fct}
\noindent
{\em Commentary.} This simply follows from quantifier elimination in $DCF_{0}$: We may assume $\phi(x_{1},x_{2},..x_{n},y)$ is a quantifier-free formula in our language of differential rings. As the $\alpha_{i}$ are constants, $\partial(t) = 1$, and $\alpha_{1}, \alpha_{2},..\alpha_{n},t$ are algebraically independent and transcendental in the  underlying field of $\mathcal{U}$, all $\phi$ can say is that the $\partial(x_{i}) = 0$, $\partial(t) = 1$, and
$P_{j}(x_{1},..,x_{n},y) \neq 0$  for finitely many polynomials $P_{j}$ over $\mathbb{Z}$. Hence we obtain the result. 
\newline
Of course this also ``follows" from strong minimality of the constants together with the ``independence" hypotheses, but our point here is just that very elementary facts are behind our proofs.

\section{Generic Painlev\'e equations $P_{II}-P_{V}$}
In this section we  prove that the solution set of each of the generic Painlev\'e equations $P_{II}-P_{V}$ is strictly disintegrated over 
$\mathbb{C}(t)$: if $y_{1},...,y_{n}$ are distinct solutions, then $y_{1},y_{1}',...,y_{n},y_{n}'$ are algebraically independent over $\mathbb{C}(t)$. 

For each of the families we will first describe the results on the classification of algebraic solutions and then make use of those results to prove strict disintegratedness. There is essentially just one common argument; using the information that in each of the families, there is a Zariski-dense subset of the parameter space for which the corresponding equation has a unique algebraic (over $\mathbb{C}(t)$) algebraic solution, together with Fact 2.11.  We will go through the details in the case of $P_{II}$, giving sketches in the remaining cases.
\subsection{The family $P_{II}$} For $\alpha\in\mathbb{C}$, $P_{II}(\alpha)$ is given by the following equation
\begin{equation*}
\partial^2 y=2y^3+ty+\alpha.
\end{equation*}
or by the equivalent Hamiltonian system:
\[S_{II}({\alpha})\left\{
\begin{array}{rll}
\partial y&=&x-y^2-\frac{t}{2}\\
\partial x&=&2xy+\alpha +\frac{1}{2}.\\
\end{array}\right.\]
It is not difficult to see that $(y,x)=(0,t/2)$ is a rational solution of $S_{II}(0)$. The work of Murata in \cite{Murata} shows that this is the only algebraic solution. However we also have ``Backlund transformations" that send solutions of  $S_{II}(\alpha)$ to that of $S_{II}(-1-\alpha)$, $S_{II}(\alpha-1)$ and $S_{II}(\alpha+1)$. We have from \cite{Murata} and \cite{Umemura2}:
\begin{fct}\label{prop1}
For $\alpha\not\in\frac{1}{2}+\mathbb{Z}$, $P_{II}(\alpha)$ has an algebraic over $\mathbb{C}(t)$ solution iff $\alpha\in\mathbb{Z}$. Furthermore, this solution is unique.
\end{fct}
We can now prove our main result:
\begin{prop}\label{P1} Let $\alpha\in\mathbb{C}$ be generic (i.e. transcendental). Then 
the solution set $X(\alpha)$ of $P_{II}(\alpha)$ is strictly disintegrated over $K = \mathbb{C}(t)$.
\end{prop}
\begin{proof}
By Fact 2.8 and Remark 2.4 it suffices to prove that any two elements of $X(\alpha)$ are mutually generic over $K_{0} = \mathbb{Q}(\alpha,t)$. 
Let $y\in X(\alpha)$ (so generic over $K_{0}$ by 2.8).  We want to show that  $acl_{X(\alpha)}(K_{0},y)=\{y\}$. For a contradiction suppose there is $z\in acl_{X(\alpha)}(K_{0},y)$, with $z\neq y$. Let the formula $\phi(\alpha,t,u,v)$ witness this, i.e. $\mathcal{U}\models \phi(\alpha,t,y,z)$ and for any $\alpha_{1},y_{1},z_{1}$ such that $\mathcal{U}\models \phi(\alpha_{1},t,y_{1},z_{1})$ we have that $z_{1}\in acl(\mathbb{Q}(\alpha_{1},t,y_{1}))$.  Now as (by 2.8) all elements of $X(\alpha)$ are generic over $K_{0}$, so by quantifier elimination they all satisfy the same formulas as $y$ over $K_{0}$.
\newline
Hence:  $\mathcal{U} \models \sigma(\alpha,t)$  where $\sigma(\alpha,t)$ is \[\forall u\left(u\in X(\alpha))\rightarrow\exists v(u\neq v\wedge v\in X(\alpha)\wedge\phi(\alpha,t,u,v))\right) \]

By Fact 2.11, $\mathcal{U} \models \sigma(\alpha_{1},t)$ for all but finitely many $\alpha_{1}\in\mathbb{C}$. So for some $n\in\mathbb{Z}$, $\sigma(n,t)$ is true in $\mathcal{U}$; that is
\[\forall u\left(u\in X(n))\rightarrow\exists v(u\neq v\wedge v\in X(n)\wedge\phi(n,t,u,v))\right).\] However, choosing $u$ to be the unique algebraic (over $\mathbb{C}(t)$) element of $X(n)$ (from \ref{prop1}), we obtain another distinct algebraic (over $\mathbb{C}(t)$) element of $X(n)$, a contradiction.  
\end{proof} 

\subsection{The family $P_{III}$} The third Painlev\'e equation is given by
\[\partial^2y=\frac{1}{y}(\partial y)^2-\frac{1}{t}\partial y+\frac{1}{t}(\alpha y^2+\beta)+\gamma y^3+\frac{\delta}{y},\] where $\alpha,\beta,\gamma,\delta\in\mathbb{C}$. However as we pointed out in \cite{NagPil}, it is enough to consider $P_{III'}(\alpha,\beta,4,-4)$ and its Hamiltonian equivalent
\[S_{III'}({v_1,v_2})\left\{
\begin{array}{rll}
\partial y&=&\frac{1}{t}(2y^2x-y^2+v_1y+t)\\
\partial x&=&\frac{1}{t}(-2yx^2+2xy-v_1x+\frac{1}{2}(v_1+v_2))
\end{array}\right.\] where $\alpha=4v_2$ and $\beta=-4(v_1-1)$. From the work of Murata \cite{Murata2} we have the following
\begin{fct}\label{prop2}
1. $S_{III'}({v_1,v_2})$ has algebraic solutions if and only if there exist an integer $n$ such that $-v_2+v_1-1=n$ or $-v_2-v_1+1=n$.\\
2. If $S_{III'}({v_1,v_2})$ has algebraic solutions, then the number of algebraic solutions is one or two. $S_{III'}({v_1,v_2})$ has two algebraic solutions if and only if there exist two integers $n$ and $m$ such that $-v_2+v_1-1=n$ and $-v_2-v_1+1=m$.
\end{fct}
From this we easily get:
\begin{prop}\label{P2}
The solution set $X(\alpha,\beta,\gamma,\delta)$ of $P_{III}(\alpha,\beta,\gamma,\delta)$, where $\alpha,\beta,\gamma,\delta\in\mathbb{C}$ are algebraically independent (and transcendental), is strictly disintegrated over $\mathbb{C}(t)$.
\end{prop}
\begin{proof}
We only need to work with $X({v_1,v_2})$ the solution set of $S_{III'}({v_1,v_2})$, $v_1,v_2$ in $\mathbb{C}$ algebraically independent. 
By Remark 2.4 again, it is enough to prove the result over $K_{0} = \mathbb{Q}(t,v_{1},v_{2})$. 
Let $y\in X({v_1,v_2})$. We want to show that  $acl_{X({v_1,v_2})}(K_{0},y)=\{y\}$. Suppose for a contradiction there is $z\in acl_{X({v_1,v_2})}(K,y)$, with $z\neq y$. As before this is witnessed by a formula $\phi(v_{1},v_{2},t,u,v)$, and again as all solutions of 
$X(v_{1},v_{2})$ are generic over $K_{0}$, the following sentence  $\sigma(v_{1},v_{2},t)$ is true in $\mathcal{U}$: 
\[\forall u\left(u\in X(v_1,v_2))\rightarrow\exists v(u\neq v\wedge v\in X(v_1,v_2)\wedge\phi(v_{1},v_{2},t, u,v))\right) \]
\\
By Fact 2.11 $\mathcal{U} \models \sigma(v_1,c,t)$ is true for all but finitely many $c\in\mathbb{C}$.
\newline
So we can find such $c$ with $-c + v_{1} - 1 \in\mathbb{Z}$. By Fact 3.3 and the fact that $v_{1}$ is transcendental, $X(v_{1},c)$ has a unique algebraic (over $\mathbb{C}(t)$) solution. As in the $P_{II}$ case we get a contradiction.
\end{proof}

\subsection{The family $P_{IV}$} 
For $\alpha,\beta\in\mathbb{C}$, the fourth Painlv\'e equation is
\[\partial^2y=\frac{1}{2y}(\partial y)^2+\frac{3}{2}y^3+4ty^2+2(t^2-\alpha)y+\frac{\beta}{y}.\] 
From the work of Murata \cite{Murata} (see also \cite{Gromak1}) we have the following:
\begin{fct}\label{prop3}
$P_{IV}$ has algebraic solutions if and only if $\alpha,\beta$ satisfy one of the following conditions:
\begin{enumerate}
\item $\alpha=n_1$ and $\beta=-2(1+2n_2-n_1)^2$, where $n_1,n_2\in\mathbb{Z}$;
\item $\alpha=n_1$, $\beta=-\frac{2}{9}(6n_2-3n_1+1)^2$, where $n_1,n_2\in\mathbb{Z}$.
\end{enumerate} Furthermore the algebraic solutions for these parameters are unique.
\end{fct}

\begin{prop}\label{P3}
The solution set $X(\alpha,\beta)$ of $P_{IV}(\alpha,\beta)$, $\alpha,\beta\in\mathbb{C}$ algebraically independent, is strictly disintegrated
over $\mathbb{C}(t)$.
\end{prop}
\begin{proof}
Again it suffices to work over $K_{0} = \mathbb{Q}(t,\alpha,\beta)$. If the conclusion fails we obtain the following sentence $\sigma(t,\alpha,\beta)$
true in $\mathcal{U}$:

\[\forall u\left(u\in X(\alpha,\beta))\rightarrow\exists v(u\neq v\wedge v\in X(\theta,\kappa)\wedge\phi(\alpha,\beta,t,u,v))\right) \]
where $\phi(\alpha,\beta,t,u,v)$ expresses that $v$ is algebraic over $\alpha,\beta,t,u$. 
\\
By Fact 2.11 (and Fact 3.5(1)) we can first choose $n_{1}\in \mathbb{Z}$, then $n_{2}\in \mathbb{Z}$ such that $\sigma(t,n_{1},n_{2})$ is true in $\mathcal{U}$ and $X(n_{1},n_{2})$ has a unique algebraic (over $\mathbb{C}(t)$) point. A contradiction as before.
\end{proof}



\subsection{The family $P_{V}$} The fifth Painlev\'e equation $P_{V}(\alpha,\beta,\gamma,\delta)$ is given by
\[\partial^2y=\left(\frac{1}{2y}+\frac{1}{y-1}\right)\left(\partial y\right)^2-\frac{1}{t}\partial y+\frac{(y-1)^2}{t^2}\left(\alpha y+\frac{\beta}{y}\right)+\gamma\frac{y}{t}+\delta\frac{y(y+1)}{y-1},\] where $\alpha,\beta,\gamma,\delta\in\mathbb{C}$. 

For our purposes it is enough to restrict to the case  when $\delta \neq 0$, in which case all algebraic (over $\mathbb{C}(t)$) solutions are rational
(see \cite{Kitaev} and \cite{Gromak1}). We let $\lambda_0=(-2\delta)^{-1/2}$, fixing $-\pi<arg(\lambda_0)<\pi$, and with the same references we have:

\begin{fct}\label{prop4}
$P_{V}$ with $\delta\neq 0$ has a rational solution if and only if for some branch of $\lambda_0$, one of the following holds with $m,n\in\mathbb{Z}$:
\begin{enumerate}
\item $\alpha=\frac{1}{2}(m+\lambda_0\gamma)^2$ and $\beta=-\frac{1}{2}n^2$ where $n>0$, $m+n$ is odd, and $\alpha\neq 0$ when $|m|<n$;
\item $\alpha=\frac{1}{2}n^2$ and $\beta=-\frac{1}{2}(m+\lambda_0\gamma)^2$ where $n>0$, $m+n$ is odd, and $\beta\neq 0$ when $|m|<n$;
\item $\alpha=\frac{1}{2}a^2$, $\beta=-\frac{1}{2}(a+n)^2$ and $\lambda_0\gamma=m$, where $m+n$ is even and $a$ arbitrary;
\item $\alpha=\frac{1}{8}(2m+1)^2$, $\beta=-\frac{1}{8}(2n+1)^2$ and $\lambda_0\gamma\not\in\mathbb{Z}$.
\end{enumerate}
\end{fct}
\begin{rmk}
In case (4) the rational solution is unique. This is also true for most of the other cases (see \cite{Kitaev}).
\end{rmk}

\begin{prop}\label{P4}
The solution set $X(\alpha,\beta,\gamma,\delta)$ of $P_{V}(\alpha,\beta,\gamma,\delta)$, $\alpha,\beta,\gamma,\delta\in\mathbb{C}$ algebraically independent, is strictly disintegrated over $\mathbb{C}(t)$. 
\end{prop}
\begin{proof} 
Assuming not, as in the earlier cases we find a sentence $\sigma(t,\alpha,\beta, \gamma, \delta)$ expressing that for any solution $u$ of $X(\alpha,\beta,\gamma,\delta)$ there is another solution $v\neq u$ which is algebraic over $\mathbb{Q}(t,\alpha,\beta,\gamma,\delta,u)$. By Fact 4.11 (applied twice), we  first find $r = \frac{1}{8}(2m+1)^2$ for some $m\in\mathbb{Z}$, and then $s = -\frac{1}{8}(2n+1)^2$ for some $n\in \mathbb{Z}$ such that $\mathcal{U}\models \sigma(t,r,s,\gamma, \delta)$, and obtain (since  $\gamma$ and $\delta$ are algebraically independent) a contradiction to 3.7(4) and 3.8.
\end{proof}

\section{Generic $P_{VI}$.}
We do not, currently, have any reason to believe that the results for generic $P_{I} - P_{V}$ do not hold for generic $P_{VI}$. But our methods, involving the description of algebraic solutions, as parameters vary, yield a weaker statement: the solution set of generic $P_{VI}$ is $\omega$-categorical, as in Definition 2.7. 


$P_{VI}(\alpha,\beta,\gamma,\delta)$, $\alpha,\beta,\gamma,\delta \in \mathbb{C}$, is given by the following equation
\begin{equation*}
\begin{split}
\partial^2y=&\frac{1}{2}\left(\frac{1}{y}+\frac{1}{y+1}+\frac{1}{y-t}\right)(\partial y)^2-\left(\frac{1}{t}+\frac{1}{t-1}+\frac{1}{y-t}\right)\partial y\\
&+\frac{y(y-1)(y-t)}{t^2(t-1)^2}\left(\alpha+\beta\frac{t}{y^2}+\gamma\frac{t-1}{(y-1)^2}+\delta\frac{t(t-1)}{(y-t)^2}\right)
\end{split}
\end{equation*}
Our result is the following:
\begin{prop}\label{P5} Let $X=X(\alpha,\beta,\delta,\gamma)$ be the solution set of $P_{VI}(\alpha,\beta,\delta,\gamma)$, where $\alpha,\beta,\delta,\gamma$ are algebraically independent, transcendental complex numbers.  Then for any $y\in X$, $acl_X(K,y)$ is finite, where $K=\mathbb{C}(t)$. Consequently as $X$ is geometrically trivial, $X$ is $\omega$-categorical.
\end{prop}

We will prove the proposition by again making use of part the classification of algebraic solution of $P_{VI}$ (see \cite{Lisovyy}). However to state the result we need, we first recall a few facts about $P_{VI}$.\\

In its hamiltonian form, $P_{VI}$ is given by
\[S_{VI}({\bar{\alpha}})\left\{
\begin{array}{rll}
\partial y&=&dH/dx\\
\partial x&=&-dH/dy
\end{array}\right.\]
where $H({\bar{\alpha},\rho})=\frac{1}{t(t-1)}\{y(y-1)(y-t)x^2-x(\alpha_4(y-1)(y-t)+\alpha_3y(y-t)+(\alpha_0-1)y(y-1))+\alpha_2(\alpha_2+\alpha_1)(y-t)\}$ and $\alpha_0+\alpha_1+2\alpha_2+\alpha_3+\alpha_4=1$. The parameters $\alpha,\beta,\delta,\gamma$ of $P_{VI}$ are related to the $\bar{\alpha}$ as follows: $\alpha=\frac{1}{2}\alpha_1^2$, $\beta=-\frac{1}{2}\alpha_4^2$, $\gamma=\frac{1}{2}\alpha^2_3$ and $\delta=\frac{1}{2}(1-\alpha^2_0)$.
\newline
Let us note that any solution $y$ of $P_{IV}(\alpha,\beta,\gamma, \delta)$  yields  a unique solution $(y,x)$ of $S_{IV}(\bar{\alpha})$. The only possible solutions $(y,x)$ of $S_{IV}(\bar{\alpha})$ not of this form are when $y=0,1,t$ and such solutions will exhibit non strong minimality of $S_{IV}(\bar{\alpha})$  (even though $P_{IV}(\alpha,\beta,\gamma, \delta)$ may be strongly minimal). 
\\

Let us note also that the relation between the parameters $\bar{\alpha}$ above and the parameters $(a_{1},a_{2},a_{3},a_{4})$ in Watanabe's Hamiltonian vector field for $P_{VI}$ (beginning of section 3 of \cite{Watanabe}) is:  $\alpha_{4} = a_{3} + a_{4}$, $\alpha_{3} = a_{3} - a_{4}$, $\alpha_{0} = 1 - a_{1} - a_{2}$, and $\alpha_{1} = a_{1} - a_{2}$. \\

We now describe some ``Backlund transformations" for $S_{VI}(\bar{\alpha})$.

\begin{table}[h]
\begin{center}
\begin{tabular}{|c|c|c|c|c|c|c|c|}
\hline
 & $\alpha_0$ & $\alpha_1$ & $\alpha_2$ & $\alpha_3$ & $\alpha_4$ & $y$ & $x$ \\
\hline
$s_0$ & $-\alpha_0$ & $\alpha_1$ & $\alpha_2+\alpha_0$ & $\alpha_3$ & $\alpha_4$ & $y$ & $x-\frac{\alpha_0}{y-t}$ \\
$s_1$ & $\alpha_0$ & $-\alpha_1$ & $\alpha_2+\alpha_1$ & $\alpha_3$ & $\alpha_4$ & $y$ & $x$ \\
$s_2$ & $\alpha_0+\alpha_2$ & $\alpha_1+\alpha_2$ & $-\alpha_2$ & $\alpha_3+\alpha_2$ & $\alpha_4+\alpha_2$ & $y+\frac{\alpha_2}{x}$ & $x$ \\
$s_3$ & $\alpha_0$ & $\alpha_1$ & $\alpha_2+\alpha_3$ & $-\alpha_3$ & $\alpha_4$ & $y$ & $x-\frac{\alpha_3}{y-1}$ \\
$s_4$ & $\alpha_0$ & $\alpha_1$ & $\alpha_2+\alpha_4$ & $\alpha_3$ & $-\alpha_4$ & $y$ & $x-\frac{\alpha_4}{y}$ \\
\hline
\end{tabular}
\caption[Some Backlund Transformations]{Some Backlund Transformations for $S_{VI}$}
\label{Table}
\end{center}
\end{table}
Recall that the Backlund transformations map solutions of a given $S_{VI}$ equation to solutions of the same equation with different values of parameters $\bar{\alpha}$, but clearly may be undefined  at certain solutions.  The list of the Backlund transformations we are interested in are given in Table \ref{Table}. The five transformations $s_0,s_1,s_2,s_3,s_4$ generate a group $\mathcal{W}$ which is isomorphic to the affine Weyl group of
type $D_4$ and which is sometimes referred to as Okamoto's affine $D_4$ symmetry group. By definition, the reflecting hyperplanes of Okamoto's affine $D_4$ action are given by the affine linear relations
\newline
$\alpha_{i} = n$  for $i=0,1,3,4$ and 
$n\in \mathbb{Z}$, 
\newline
as well as 
\newline 
$\alpha_0\pm\alpha_1\pm\alpha_3\pm\alpha_4=2n+1$ for $n\in \mathbb{Z}$. 

Let $\mathcal{M}$ be the union of all these hyperplanes. Then as proven in \cite{Watanabe} by Watanabe (see Theorem 2.1(v)), and discussed in \cite{NagPil}, if $(\alpha_0,\alpha_1,\alpha_3,\alpha_4)\not\in\mathcal{M}$ then the solution set of $S_{VI}(\bar{\alpha})$ is strongly minimal (equivalently Umemura's $J$-condition holds). 

\begin{rem} 
(i) Let $t_{0} = s_{0}s_{2}(s_{1}s_{3}s_{4}s_{2})^{2}$, $t_{1} = s_{1}s_{2}(s_{0}s_{3}s_{4}s_{2})^{2}$, $t_{3} = s_{3}s_{2}(s_{0}s_{1}s_{4}s_{2})^{2}$, and $t_{4} = s_{4}s_{2}(s_{0}s_{1}s_{3}s_{2})^{2}$.  Then for $i=0,1,3,4$, and parameters $(\alpha_{0},\alpha_{1},\alpha_{3}\alpha_{4})$, $t_{i}(\alpha_{i}) = \alpha_{i}-2$, and $t_{i}(\alpha_{j}) = \alpha_{j}$.  Hence the orbit of $(\alpha_{0},\alpha_{1},\alpha_{3},\alpha_{4})$ under
$\mathcal{W}$ includes  $\{(\alpha_{0}-2\mathbb{Z}, \alpha_{1}-2\mathbb{Z}, \alpha_{3}-2\mathbb{Z}, \alpha_{4}-2\mathbb{Z})\}$. 
\newline
(ii) If $(\alpha_{0},\alpha_{1},\alpha_{3},\alpha_{4}) \notin \mathcal{M}$ then its orbit under $\mathcal{W}$ also avoids $\mathcal{M}$.
\newline
(iii) If $\bar{\alpha} = (\alpha_{0},\alpha_{1},\alpha_{3},\alpha_{4}) \notin \mathcal{M}$, then each $s_{i}$, $i=0,1,2,3,4$ establishes a {\em bijection} between the  solutions of  $S_{VI}(\bar{\alpha})$ and  $S_{VI}(s_{i}(\bar{\alpha}))$.
\end{rem}
\begin{proof} (iii) This is because the solution sets of both $S_{VI}(\bar{\alpha})$  and $S_{VI}(s_{i}(\bar{\alpha}))$  are strongly minimal, hence neither has a solution of form $(y,x)$ with $y=0,1$ or $t$, or $x=0$. So $s_{i}$ is defined on all solutions. Using the fact that $s_{i}^{2}$ is the identity for each $i$ we obtain the desired conclusion. 
\end{proof}

The key result is Boalch's ``generic icosahedral solution": see Section 6 of \cite{Boalch}.
\begin{fct}
The equation $S_{VI}(1/2, -1/5, 1/3, 2/5)$ has exactly $12$ algebraic solutions  (of course all in $\mathbb{Q}(t)^{alg}$). Moreover
$(1/2, 4/5, 1/3, 2/5)\notin \mathbb{M}$. 
\end{fct}

By Remark 4.2 we conclude:
\begin{cor}
Let $\bar{\alpha} \in \{(1/2-2\mathbb{Z},-1/5-2\mathbb{Z},1/3-2\mathbb{Z},2/5-2\mathbb{Z})$. Then  $S_{VI}(\bar{\alpha})$ has precisely $12$ algebraic solutions (again necessarily in $\mathbb{Q}(t)^{alg}$). 
\end{cor}
This is enough for us to prove our result:
\begin{proof}[Proof of Proposition \ref{P5}]
For algebraically independent transcendental $\alpha, \beta, \gamma, \delta \in \mathbb{C}$. The solutions of $P_{IV}(\alpha,\beta,\gamma,\delta)$ are in bijection with those of the $S_{VI}(\alpha_{0},\alpha_{1},\alpha_{3},\alpha_{4})$ (where the $\alpha_{i}$ are related to $\alpha,\beta, \gamma, \delta$ as stated above) via  $y \to (y,x)$. (Because $\alpha_{0},\alpha_{1},\alpha_{3},\alpha_{4}$ are also algebraically independent, so $\bar{\alpha}\notin \mathbb{M}$ and $S_{VI}(\bar{\alpha})$ is strongly minimal.) Without ambiguity we denote a solution of a system $S_{VI}(-)$ by $y$.
Let now $X(\bar{\alpha})$ denote the solution set of  $S_{VI}(\bar{\alpha})$  (likewise for other parameters). Let $y\in X(\bar{\alpha})$. As before (using Remark 2.4), it is enough to work over $K_{0} = \mathbb{Q}(\bar{\alpha}, t)$. We know $y$ (like all elements of $X(\bar{\alpha})$) is generic over $K_{0}$.
\newline
{\em Claim.} $acl_{X(\bar{\alpha})}(K_{0},y)$ has cardinality at most $12$ (including $y$ itself). 
\newline
{\em Proof.} The same argument as before: if not then we find a true sentence $\sigma(\alpha_{0},\alpha_{1},\alpha_{3},\alpha_{4},t)$ expressing that for any solution $y$ of $S_{VI}(\bar{\alpha})$, there are at least $12$ other solutions in the algebraic closure of $\mathbb{Q}(\bar{\alpha},t,y)$
(i.e. in $(\mathbb{Q}(\bar{\alpha},t,y,y')^{alg}$). Applying Fact 2.11, we find (one by one), $(r_{0},r_{1},r_{3},r_{4})\in \{(1/2-2\mathbb{Z},-1/5-2\mathbb{Z},1/3-2\mathbb{Z},2/5-2\mathbb{Z})\}$ such that $\mathbb{U} \models \sigma(r_{0},r_{1},r_{3},r_{4},t)$. But then choosing $y$ to be one of the algebraic solutions of $S_{VI}(r_{0},r_{1},r_{3},r_{4})$ we obtain at least $12$ other algebraic solutions, contradicting Corollary 4.4. This proves the claim and the Proposition.
\end{proof}

\end{document}